\documentclass[11pt, a4paper]{amsart}

\usepackage{amssymb,amsfonts, amsmath, amsthm}
\usepackage[all,arc]{xy}
\usepackage{enumerate}
\usepackage{mathrsfs}
\usepackage{mathtools}
\usepackage{enumitem}
\usepackage[mathscr]{euscript}
\usepackage{array}
\usepackage{tikz}
\usepackage{tikz-cd}
\usepackage{textcomp}
\usepackage[pagebackref, colorlinks=true, linkcolor=blue, citecolor = red]{hyperref}
\usetikzlibrary{patterns}
\usepackage[nameinlink,capitalise,noabbrev]{cleveref}
\crefname{subsection}{Subsection}{Subsection}

\usepackage[margin=2cm]{geometry}

\renewcommand*{\backref}[1]{}
\renewcommand*{\backrefalt}[4]{({%
		\ifcase #1 Not cited.%
		\or On p.~#2%
		\else On pp.~#2%
		\fi%
	})}
    
\newcommand{\sC}{\mathscr{C}}
\newcommand{\cF}{\mathcal{F}}
\newcommand{\bN}{\mathbb{N}}
\newcommand{\sX}{\mathscr{X}}

\newcommand{\Open}{\mathscr{O}\mathrm{pen}}
\newcommand{\corp}{\mathrm{corp}}
\newcommand{\op}{\mathrm{op}}
\newcommand{\Map}{\mathrm{Map}}
\newcommand{\Fun}{\mathrm{Fun}}
\newcommand{\colim}{\operatorname*{colim}}
\newcommand{\Shv}{\mathscr{S}\mathrm{hv}}
\newcommand{\Cond}{\mathscr{C}\mathrm{ond}}
\newcommand{\An}{\mathscr{A}\mathrm{n}}
\newcommand{\Top}{\mathscr{T}\mathrm{op}}
\newcommand{\Set}{\mathscr{S}\mathrm{et}}
\newcommand{\CHaus}{\mathscr{C}\mathscr{H}\mathrm{aus}}

\newcommand{\ProFin}{\mathscr{P}\mathrm{ro}\mathscr{F}\mathrm{in}}
\newcommand{\ExtrDisc}{\mathscr{E}\mathrm{xtr}\mathscr{D}\mathrm{isc}}
\newcommand{\Clopen}{\mathscr{C}\mathrm{lopen}}
\newcommand{\open}{\mathrm{open}}
\newcommand{\ad}{\mathrm{ad}}
\newcommand{\id}{\mathrm{id}}

\newcommand{\inj}{\mathrm{inj}}

\newcommand{\Sl}{\mathscr{S}\mathrm{l}}

\newcommand{\yo}{\text{\usefont{U}{min}{m}{n}\symbol{'210}}}

\DeclareFontFamily{U}{min}{}
\DeclareFontShape{U}{min}{m}{n}{<-> udmj30}{}

\newtheorem{thm}[equation]{Theorem}
\newtheorem{lemma}[equation]{Lemma}
\newtheorem{prop}[equation]{Proposition}
\newtheorem{cor}[equation]{Corollary}

\newtheorem{mainthm}{Theorem}

\theoremstyle{definition}
\newtheorem{dfn}[equation]{Definition}
\newtheorem{ex}[equation]{Example}

\theoremstyle{remark}
\newtheorem{rem}[equation]{Remark}

\newtheorem{conj}[equation]{Conjecture}
\newtheorem{notn}[equation]{Notation}

\numberwithin{equation}{section}

\title{Fractured Structures in Condensed Mathematics}

\date{March 2026}
\author{Nima Rasekh}
\address{Institut f{\"u}r Mathematik und Informatik, Universit{\"a}t Greifswald, Greifswald, Germany}
\email{nima.rasekh@uni-greifswald.de}

\author{Qi Zhu}
\address{Max Planck Institute for Mathematics, 53111 Bonn, Germany}
\email{qzhu@mpim-bonn.mpg.de}


\begin{document}

\begin{abstract}
	We construct a fractured structure, in the sense of Lurie \cite{lurie2018sag}, on the $\infty$-topos of condensed anima. This fractured structure allows us to better comprehend various properties of condensed anima -- we use it to exhibit an explicit collection of jointly conservative points for $\Cond(\An)$. To rule out further candidates for fractured structures, we analyze limits in the category of extremally disconnected spaces $\ExtrDisc$. In particular, we show that $\ExtrDisc$ does not admit all fibers, answering a question from Clausen.
\end{abstract}

\maketitle
\tableofcontents

\vspace{-1em}
\begin{figure}[ht!]
\centering
\includegraphics[width=64mm]{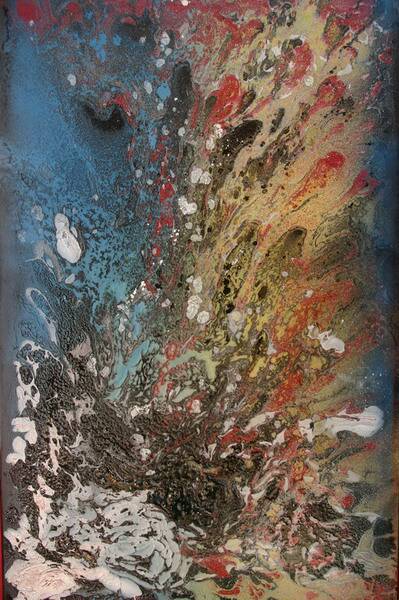}
\caption{\emph{The Fractured Mind} by Rob Crane from \href{https://newirishart.com/artworks/rob-crane-the-fractured-mind-rc0109.htm}{New Irish Art}.}
\end{figure}
\newpage

\section{Introduction}

\addtocontents{toc}{\protect\setcounter{tocdepth}{1}}

\subsection{From petit vs.~gros to fractured topoi} \label{subsec:gros and petit}
Topos theory was introduced as a categorical method to study topologies in a variety of settings. However, somewhat surprisingly, it permitted two parallel approaches. On the one hand, there are topoi which capture the data of one particular space, such as the topos of sheaves on a topological space or the étale topos of a scheme. On the other hand, there are topoi which include many spaces at once, such as sheaves on the site of all topological spaces or sheaves on the site of all schemes. Motivated by their French origin, these two classes of topoi are often informally referred to as \emph{petit} and \emph{gros} topoi, respectively. 

\medskip \noindent As a result, right from the start a central theme in topos theory has been to bridge these two parallel approaches. For example, already in \cite[Expos\'e VII, Proposition 4.1]{grothendieck1972sga4tome2} the authors observe that {\'e}tale sheaves on a scheme $X$ sit fully faithfully inside the big {\'e}tale topos sliced over $X$, resulting in the following crucial observation \cite[Exposé VII, under Proposition 4.1]{grothendieck1972sga4tome2}:
\begin{quote}
    \emph{Pour l'\'etude de la cohomologie des faisceaux, il est essentiellement \'equivalent de travailler avec le petit site \'etale $X_{\mathrm{\acute{e}t}}$, ou le gros site \'etale $(\mathrm{Sch}/X)_{\mathrm{\acute{e}t}}$.}
\end{quote}
They establish a similar relationship between the petit and gros topos of a topological space \cite[Expos\'e IV, 4.10.4]{grothendieck1972sga4tome1}. Other examples of this phenomenon include the Zariski site \cite[\href{https://stacks.math.columbia.edu/tag/020Z}{Tag 020Z}]{stacksproject2026}, the crystalline site \cite[\href{https://stacks.math.columbia.edu/tag/07IJ}{Tag 07IJ}]{stacksproject2026}, and the pro-{\'e}tale site \cite[\href{https://stacks.math.columbia.edu/tag/098Q}{Tag 098Q}]{stacksproject2026}, among others.

\medskip \noindent It is hence unsurprising that various efforts have been made to axiomatize the relation between petit and gros topoi. One successful method is the theory of \emph{fractured $\infty$-topoi} due to Lurie \cite{lurie2018sag}, motivated by Carchedi \cite{carchedi2020topoi}.\footnote{Another related approach is given by \emph{cohesive $\infty$-topoi}, which will not be our focus here. See \cref{rem: fractured topoi via cohesive topoi} for a discussion of the relationship between these two approaches.} A \emph{fractured structure} on an $\infty$-topos $\sX$ is a (not necessarily full) subcategory $\sX^{\corp}$, such that the inclusion functor admits a right adjoint satisfying various axioms. See \cref{dfn: fractured topos} for a detailed definition. The conditions in particular imply that for an object $X$ in $\sX^{\corp}$, the slice category $\sX^{\corp}_{/X}$ sits fully faithfully inside $\sX_{/X}$ (\cref{rem:ff}). Thus, thinking intuitively of $\sX$ as a gros topos and $\sX^{\corp}$ as a petit topos, the fractured structure axiomatically realizes the anticipated relationship between gros and petit topoi in a very general setting. 

\medskip \noindent As one might expect, many situations where gros and petit topoi occur are in fact examples of fractured $\infty$-topoi. This includes some of the aforementioned examples, such as the classical Zariski site and \'etale site \cite[Examples 20.6.4.1 and 20.6.4.2]{lurie2018sag}, but also their analogues from spectral algebraic geometry \cite[Examples 20.6.4.4 and 20.6.4.5]{lurie2018sag}. Beyond algebraic geometry, further examples of fractured structures have been exhibited in differential geometry \cite{clough2024diff1, clough2024diff2}, global equivariant homotopy theory \cite{cnossenlenzlinskens2025partial}, and the theory of six functor formalisms \cite[Section 7.3]{cnossenhoyois2026htt}.

\subsection{Condensed anima as a gros topos} \label{subsec:condensed intro}
The idea of a topological space is ubiquitous in mathematics and plays a prominent role in algebra, for instance via topological abelian groups. These are notoriously ill-behaved, and in particular do not form an abelian category, a crucial requirement to perform homological algebra. This has impeded progress in areas in which topological groups naturally occur, such as analytic geometry. Solving this issue resulted in the theory of \emph{condensed mathematics}, developed by Clausen and Scholze \cite{scholze2019condensed}, and independently by Barwick and Haine \cite{barwickhaine2019pyknotic}, leading to successful applications in analytic geometry \cite{clausenscholze2020analytic}.

\medskip \noindent Fundamentally, the success of condensed mathematics is due to the power of topos theory. By isolating a desirable class of topological spaces, in this case extremally disconnected compact Hausdorff spaces, denoted $\ExtrDisc$, and then passing to sheaves on them, one obtains the $\infty$-topos of \emph{condensed anima} 
\[ \Cond(\An) \coloneqq \Shv_{\An}(\ExtrDisc). \] 
We will give a more detailed review in \cref{subsec:condensed anima}. This construction simultaneously includes examples of interest, in particular all compactly generated topological spaces, while gaining many desirable properties inherent to sheaf categories. Historically, such methods can already be seen in the work of Johnstone, who introduced the topological topos with similar aims \cite{johnstone1979topologicaltopos}. The theory has risen to prominence in the last few years and has found applications in various areas of mathematics, such as in algebra \cite{aoki2024semitopologicalktheorysolidification, brink2025condensedgroupcohomology}, analysis \cite{scholze2019condensed}, geometry \cite{clausenscholze2020analytic,clausenscholze2022condensed,haine2025condensedhomotopytypescheme}, representation theory \cite{heyermann2024sixfunctorformalismssmoothrepresentations}, set theory \cite{bergfalk2025whiteheadsproblemcondensedmathematics} and topology \cite{mor2023picardbrauergroupsknlocal,clausen2025dualitylinearizationpadiclie,mair2025thesis}.

\medskip \noindent Condensed anima include all (compactly-generated) topological spaces \cite[Proposition 1.7]{scholze2019condensed}. Hence, from the perspective of \cref{subsec:gros and petit}, condensed anima exhibit the behavior of a gros topos. It is therefore natural to ask whether there is a subcategory of condensed anima that gives a fractured structure, making our intuition regarding condensed anima as a gros topos into a rigorous statement. 

\subsection{Fractured structure on condensed anima}
In this article, we realize this vision and construct a fractured structure on condensed anima. With $\Cond(\An) = \Shv(\ExtrDisc)$, we choose the collection of open embeddings on extremally disconnected spaces as a base for the corresponding petit topos. Let $\ExtrDisc^{\open}$ denote the wide subcategory of extremally disconnected spaces spanned by the open embeddings, equipped with a Grothendieck topology given by finitely jointly surjective families of open embeddings. Our main theorem is the following. 
\begin{mainthm}[\cref{thm:main}]
	The left Kan extension of the restriction functor
	\[ \Cond^{\open}(\An) \coloneqq \Shv(\ExtrDisc^{\open}) \longrightarrow \Cond(\An), \] 
	induces a fractured $\infty$-topos structure.
\end{mainthm}
\noindent More precisely, we mean that this functor induces an equivalence onto a subcategory of $\Cond(\An)$, and the inclusion of this subcategory is a fractured structure.

\medskip \noindent The philosophy of petit and gros topoi is laid out clearly after passing to slice topoi by the following result. Let $\yo_{\open} \colon \ExtrDisc \to \Cond^{\open}(\An)$ be the sheafification of the Yoneda embedding.\footnote{Here, the sheafification step is in fact non-trivial. See \cref{rem:extrdisc not subcanonical} for further details.}

\begin{mainthm}[{\cref{prop:condopen over yoS}}]
	For every $S \in \ExtrDisc$ there is an equivalence 
	\[ \Cond^{\open}(\An)_{/\yo_{\open} S} \simeq \Shv(S) \]
	of $\infty$-categories.
\end{mainthm}

\noindent The fractured structure along with the above computation allows us to draw a number of formal consequences and better understand condensed anima. For example, following Clough \cite{clough2024diff1}, the framework gives us a new constructive proof for the fact that $\Cond(\An)$ has enough points, and even an explicit collection of such points.

\begin{mainthm}[\cref{thm:formula enough points}]
	The collection of points
	\[ \left\{ \Cond(\An) \longrightarrow \An, \ F \mapsto \colim_{\substack{U \subseteq S \ \mathrm{clopen} \\ s \in U}} F(U) \right\}_{S \in \ExtrDisc, \ s \in S} \]
	is jointly conservative.
\end{mainthm}

\subsection{The magical topology of extremally disconnected spaces} The choice of open embeddings on extremally disconnected spaces for our petit topos was far from arbitrary. Our last main result aims to show that the construction and study of the fractured structure on $\Cond(\An)$ relies on very specific point-set topological properties of extremally disconnected spaces. 

\medskip \noindent A priori, there are various other constructions one could envision to build fractured structures on condensed anima. One could change the choice of embeddings and even enlarge the base $\ExtrDisc$ to profinite sets or compact Hausdorff spaces, but we show in \cref{corollary: no fractured structure CHaus ProFin} that enlarging the base is not a fruitful approach.

\medskip \noindent Staying within the realm of extremally disconnected spaces, a natural replacement of open embeddings appears to be the collection of all injections. However, we prove in \cref{corollary: ExtrDiscinj not admissibility} that this too does not result in a fractured structure. The main culprit lies in ill-behaved features of the category $\ExtrDisc$ that are imposed by the point-set topological features of extremally disconnected spaces. More precisely, the failure lies in the following conjecture proposed by Clausen (\cref{conj:no fibers in extrdisc}), which we prove.
\begin{mainthm}[\cref{thm:no fibers}]
	The category $\ExtrDisc$ does not admit all fibers.
\end{mainthm}

\noindent Indeed, let $\beta \pi_1 \colon \beta(\bN \times \bN) \to \beta \bN$ be induced by the projection onto the first coordinate. We show that the fiber over an element $p \in \beta \bN \setminus \bN$ is not extremally disconnected.

\subsection*{Acknowledgements}
This article grew out of the second author's master's thesis, which was supervised by the first author. We thank Adrian Clough for helpful discussions, which ultimately motivated this project. Moreover, we thank Dustin Clausen, Catrin Mair and Peter Scholze for insightful conversations. We are also grateful to the Max Planck Institute for Mathematics in Bonn for its hospitality and financial support.
\addtocontents{toc}{\protect\setcounter{tocdepth}{5}}

\section{Fractured topoi and condensed anima}
\noindent In this section we primarily review the notion of fractured topoi and condensed anima. We also discuss the relation of fractured topoi to comonadic structures (\cref{lemma:monadic}) and cohesive $\infty$-topoi (\cref{rem: fractured topoi via cohesive topoi}).

\subsection{Fractured topoi} \label{subsection: fractured topoi}
Following the discussion in \cref{subsec:gros and petit}, we now review the notion of fractured $\infty$-topoi.

\begin{dfn}[{\cite[Definition 20.1.2.1]{lurie2018sag}}] \label{dfn: fractured topos}
    Let $\sX$ be an $\infty$-topos. A subcategory $j_! \colon \sX^{\corp} \to \sX$ is a \emph{fracture subcategory} if it satisfies the following conditions:
    \begin{enumerate}[label=(\roman*), leftmargin=*]
        \item If $X \in \sX^{\corp}$ and $f \colon X \to Y$ in $\sX$ is an equivalence, then $f$ belongs to $\sX^{\corp}$.
        \item The $\infty$-category $\sX^{\corp}$ admits pullbacks and these are preserved by $j_!$.
        \item The inclusion functor $j_! \colon \sX^{\corp} \to \sX$ admits a right adjoint $j^* \colon \sX \to \sX^{\corp}$ which is conservative and preserves small colimits.
        \item For every map $U \to V$ in $\sX^{\corp}$ the diagram
        \[
            \begin{tikzcd}
                j_! j^* j_! U \arrow[r] \arrow[d] & j_! j^* j_! V \arrow[d]
                \\ j_! U \arrow[r] & j_! V
            \end{tikzcd}
        \]
        induced by the counit $j_! j^* \Rightarrow \id_{\sX}$ is a pullback in $\sX$.
    \end{enumerate}
    A \emph{fractured $\infty$-topos} is a pair $\sX^{\corp} \to \sX$ where $\sX$ is an $\infty$-topos and $\sX^{\corp}$ is a fracture subcategory of $\sX$. Objects in $\sX^{\corp}$ are called \emph{corporeal}.
\end{dfn}

\noindent An equivalent definition can be found in \cite[Proposition 1.14]{clough2024diff1}. 

\begin{rem} \label{rem:right adjoint}
	Since $j^*$ preserves small colimits and $j_! \colon \sX^{\corp} \to \sX$ is a map between presentable $\infty$-categories, it admits a right adjoint $j_*$ by the adjoint functor theorem.
\end{rem}

\begin{rem} \label{rem:ff}
    For every object $X \in \sX^{\corp}$ the functor $(j_!)_{/X} \colon \sX^{\corp}_{/X} \to \sX_{/X}$ is fully faithful by \cite[Proposition 20.1.3.1]{lurie2018sag}. Hence, fractured $\infty$-topoi precisely realize the intuition regarding petit and gros topoi discussed in \cref{subsec:gros and petit}, where $\sX^{\corp}_{/X}$ plays the role of the petit topos and $\sX_{/X}$ the role of the gros topos.
\end{rem}

\noindent Fractured $\infty$-topoi have many interesting properties. See \cite[Section 20.1.3]{lurie2018sag} for more details. We do establish the following (co)monadicity statement.
\begin{lemma} \label{lemma:monadic}
    Let $j_! \colon \sX^{\corp} \to \sX$ be a fractured $\infty$-topos. The adjunction $(j_!,j^*)$ is monadic and the adjunction $(j^*,j_*)$ is comonadic. Hence, $j^*$ creates all limits and colimits.
\end{lemma}

\begin{proof}
    As explained in \cite[Remark 20.1.3.8]{lurie2018sag}, $j^*$ is conservative and preserves all limits as a right adjoint, and so monadicity follows from the Barr--Beck theorem \cite[Theorem 4.7.3.5]{lurie2017ha}. Similarly, Barr--Beck applied to $(j^*,j_*)$ implies comonadicity. The last statement is a standard consequence of (co)monadic adjunctions \cite[Theorem 5.7]{riehlverity2015completeness}.
\end{proof}

\noindent Beyond \cref{lemma:monadic} we make one additional observation. Recall from the introduction (\cref{subsec:gros and petit}) that there is another effort to axiomatize the relation between petit and gros topoi, known as \emph{cohesive topos theory}, first pioneered by Lawvere \cite{lawvere1994cohesivetopos, lawvere2007cohesion} and later developed by Schreiber \cite{schreiber2013cohesive, sati2020proper}. 

\medskip \noindent It is well-known that condensed anima are not compatible with the cohesive approach \cite[Example 2.2.14]{barwickhaine2019pyknotic}, and hence we will not pursue this approach further. However, we observe the following non-trivial interaction between cohesive and fractured $\infty$-topoi.

\begin{rem} \label{rem: fractured topoi via cohesive topoi}
    Let $j_! \colon \sX^{\corp} \to \sX$ be a fractured $\infty$-topos and let $X \in \sX^{\corp}$. So we have a triple adjunction $j_! \dashv j^* \dashv j_*$ by \cref{rem:right adjoint}. Passing to slice categories yields a triple adjunction
    \[
      \begin{tikzcd}
        \sX^{\corp}_{/X} \arrow[rr, "(j_!)_{/X}", shift left = 4, "" name = 1] \arrow[rr, "(j_*)_{/X}", swap, shift right = 4, ""name = 3] & & \sX_{/X} \arrow[ll, "(j^*)_{/X}" description, "" name = 2]
        \arrow[phantom,from=1,to=2,"\scriptscriptstyle\bot", yshift = 2]
        \arrow[phantom,from=2,to=3,"\scriptscriptstyle\bot", yshift = 1.5]
      \end{tikzcd}
    \]
        In this setting, $(j_!)_{/X}$ is fully faithful (\cref{rem:ff}) and preserves pullbacks. Moreover, $\id_X \in \sX_{/X}^{\corp}$ by \cref{dfn: fractured topos}(i), so it is a terminal object and hence $(j_!)_{/X}$ preserves terminal objects. Thus $(j_!)_{/X}$ preserves finite limits.
        
        \medskip \noindent These are almost enough conditions to provide a further left adjoint of $(j_!)_{/X}$ which would result in a quadruple adjunction. If this left adjoint would then also preserve finite products, then we would recover a cohesive structure in the sense of Lawvere--Schreiber \cite[Definition 3.4.1, Remark 3.4.2]{schreiber2013cohesive}.
\end{rem}

\noindent A fractured $\infty$-topos is in general difficult to produce by hand. Machinery developed in \cite{lurie2018sag} allows us to build fractured $\infty$-topoi from site-theoretic notions.

\begin{dfn}[{\cite[Definition 20.2.1.1, 20.6.2.1]{lurie2018sag}}] \label{definition: admissibility}
    Let $(\sC, \tau)$ be an $\infty$-site.
    \begin{enumerate}[label=(\roman*), leftmargin=*]
        \item An \emph{admissibility structure} on $\sC$ is a collection of maps in $\sC$, called the \emph{admissibility morphisms}, satisfying the following conditions:
        \begin{enumerate}[label=(\alph*), leftmargin=*]
            \item Equivalences in $\sC$ are admissible.
            \item If $f \colon U \to X$ is admissible in $\sC$ and $g \colon X' \to X$ is a map in $\sC$, then the pullback square
            \[
                \begin{tikzcd}
                    U \times_X X' \arrow[r] \arrow[d, "f'", swap] \arrow[dr, phantom, very near start, "\lrcorner"]  & U \arrow[d, "f"] \\ 
                    X' \arrow[r, "g", swap] & X
                \end{tikzcd}
            \]
            exists and $f'$ is admissible.
            \item Let 
            \[
                \begin{tikzcd}
                    X \arrow[rr, "f"] \arrow[dr, "h", swap] & & Y \arrow[dl, "g"] \\ 
                    & Z
                \end{tikzcd}
            \]
            be a commuting triangle in $\sC$ with admissible $g$. Then, $f$ is admissible if and only if $h$ is admissible.
            \item The collection of admissible morphisms is closed under retracts in $\Fun([1], \sC)$.
        \end{enumerate}
        We write $\sC^{\ad}$ for the wide subcategory of $\sC$ spanned by the admissible morphisms.
        \item An admissibility structure $\sC^{\ad} \to \sC$ is \emph{compatible with $\tau$} if for every $X \in \sC$ and covering sieve $\sC_{/X}^{(0)} \to \sC_{/X}$ there exists a $\tau$-covering\footnote{By this we mean a collection of maps that generate a covering sieve, see \cite[Section 20.6.1]{lurie2018sag}.} $\{f_{\alpha} \colon U_{\alpha} \to X \}_{\alpha}$ such that $f_{\alpha}$ is admissible and belongs to $\sC_{/X}^{(0)}$.
        \item A \emph{geometric site} is a triple $(\sC, \sC^{\ad}, \tau)$ where $(\sC, \tau)$ is an $\infty$-site and $\sC^{\ad}$ is an admissibility structure on $\sC$ compatible with $\tau$.
    \end{enumerate}
\end{dfn}

\noindent Let $j \colon \sC^{\ad} \to \sC$ be part of a geometric site $(\sC, \sC^{\ad}, \tau)$. There is a Grothendieck topology $\tau^{\ad} \coloneqq j^* \tau$ on $\sC^{\ad}$ characterized by: For any $C \in \sC^{\ad}$ a sieve $(\sC^{\ad}_{/C})^{(0)} \subseteq \sC^{\ad}_{/C}$ is a covering sieve if and only if the collection $\{F(U) \to F(C) \}_{U \in (\sC^{\ad}_{/C})^{(0)}}$ is a $\tau$-covering of $F(C)$, see \cite[Proposition 20.6.1.1]{lurie2018sag}.

\begin{thm}[{\cite[Theorem 20.6.3.4]{lurie2018sag}}]
    Let \label{theorem: geometric site} $(\sC, \sC^{\ad}, \tau)$ be a geometric site. Precomposition with $j$ induces an adjunction
    \[
       \begin{tikzcd}
         \Shv_{\tau^{\ad}}(\sC^{\ad}) \arrow[rr, "j_!", shift left = 1.5, ""name = 1] & & \Shv_{\tau}(\sC) \arrow[ll, "j^*", shift left = 1.5, ""name = 2]
         \arrow[phantom,from=1,to=2,"\scriptscriptstyle\bot"]
       \end{tikzcd}
     \]
        and $j_!$ induces an equivalence from $\Shv_{\tau^{\ad}}(\sC^{\ad})$ onto a fracture subcategory $\Shv_{\tau}(\sC)^{\corp} \subset \Shv_{\tau}(\sC)$.
\end{thm}

\noindent One application of fractured $\infty$-topoi lies in the study of points and hence hypercompleteness of $\infty$-topoi. Recall that a \emph{point} of an $\infty$-topos $\sX$ is a geometric morphism $x_* \colon \An \to \sX$. Then, $\sX$ is said to \emph{have enough points} if there exists a family of points $x_* \colon \An \to \sX$, such that the collection of associated left adjoints $\{x^* \colon \sX \to \An : x \text{ point} \}$ is jointly conservative.
\begin{prop}[{\cite[Proposition 1.5]{clough2024diff1}}] \label{prop: enough points}
    Let $j_! \colon \sX^{\corp} \to \sX$ be a fractured $\infty$-topos with right adjoint $j^*$. Let $\{x_i \colon \An \to \sX^{\corp} \}_i$ be a collection of jointly conservative points in $\sX^{\corp}$. Then, $\{ j_* \circ x_i \colon \An \to \sX \}_i$ is a collection of jointly conservative points in $\sX$, where $j_*$ is the right adjoint to $j^*$.
\end{prop}

\begin{proof}
    This follows immediately from the conservativity of $j^*$.
\end{proof}

\begin{rem}
 In particular, $\sX$ has enough points if $\sX^{\corp}$ has enough points.
\end{rem}

\begin{cor} \label{cor:hypercomplete}
    Let $j_! \colon \sX^{\corp} \to \sX$ be a fractured $\infty$-topos such that $\sX^{\corp}$ has enough points. Then, $\sX$ is hypercomplete.
\end{cor}

\begin{proof}
    This follows immediately from \cref{prop: enough points} since $\infty$-topoi with enough points are hypercomplete \cite[Remark 6.5.4.7]{lurie2009htt}.
\end{proof}

\subsection{Condensed mathematics}\label{subsec:condensed anima}
Continuing the explanation in \cref{subsec:condensed intro}, condensed anima should be viewed as nice sheaves on topological spaces. We only briefly recall the setup to familiarize the reader with the notation. We refer the reader to \cite{scholze2019condensed, heyermann2024sixfunctorformalismssmoothrepresentations, mair2025thesis} for more detailed treatments.

\medskip \noindent Consider the chain of full subcategories
\[
    \begin{tikzcd}
        \ExtrDisc \arrow[r, hookrightarrow] & \ProFin \arrow[r, hookrightarrow] & \CHaus.
    \end{tikzcd}
\]
of the (1-)category of compact Hausdorff spaces $\CHaus$. We recall that the (1-)category of extremally disconnected spaces $\ExtrDisc$ consists of those compact Hausdorff spaces such that the closure of an open set is open again. Now, we can endow each of these with the Grothendieck topology of finitely jointly surjective families. Taking only $\kappa$-small spaces for an uncountable strong limit cardinal $\kappa$, we obtain sites and these inclusions induce equivalences on hypersheaves by an application of a base comparison result \cite[Proposition 3.12.11]{barwickglasmanhaine2018exodromy}.

\begin{dfn}[{\cite[Definition 11.7]{clausenscholze2020analytic}}]
    The $\infty$-topos of \emph{condensed anima} is defined as
    \[ \Cond(\An) \coloneqq \Shv_{\An}(\ExtrDisc_{\kappa}). \]
\end{dfn}

\noindent In the following, we will ignore the $\kappa$, as it will not be relevant for our considerations, see also \cite[Remark 3.5.6]{heyermann2024sixfunctorformalismssmoothrepresentations}.

\begin{rem} \label{remark: ExtrDisc categorically bad}
    We warn that the category $\ExtrDisc$ behaves extremely\footnote{Pun intended.} badly. This has caused us many headaches throughout the project. For example, not even binary products of infinite extremally disconnected spaces exist in $\ExtrDisc$, see \cite[Proposition 2.54]{land2022condensed}. We will later show that it does not admit all fibers, see \cref{thm:no fibers}. Let us moreover remark that every convergent sequence in extremally disconnected spaces is eventually constant \cite[Theorem 1.3]{Gleason1958ProjectiveTS}. These objects are intimately connected to set theory since every extremally disconnected space is a retract of the Stone--\v{C}ech compactification of a discrete space \cite[Corollary]{rainwater1959projective}. The latter often features in set-theoretic contexts.
\end{rem}

\begin{rem}\label{rem:extrdisc projective}
  One primary reason to work with extremally disconnected spaces is that they are precisely the projective objects in $\CHaus$ by Gleason \cite[Theorem 2.5]{Gleason1958ProjectiveTS}, meaning every continuous surjection into an extremally disconnected space admits a section.
\end{rem}

\begin{lemma}[{\cite[Section 2.2.9]{barwickhaine2019pyknotic}}] \label{lemma: condensed anima is finite product preserving functor}
    A functor $\ExtrDisc^{\op} \to \An$ is a condensed anima if and only if it preserves finite products.
\end{lemma}

\begin{ex} \label{ex:subcanonical}
    Let $X$ be a topological space. Then, the restricted representable presheaf
    \[ \yo X \colon \ExtrDisc^{\op} \to \An, \ Y \mapsto \Map_{\Top}(Y, X) \]
    is a condensed anima, as $\Map_{\Top}(-,X)$ preserves finite products. In particular, the Grothendieck topology on 
    $\ExtrDisc$ is subcanonical.
\end{ex}

\section{A fractured structure on condensed anima} \label{sec:fractured}
\noindent In this section we prove the first main result, i.e.,~we provide a fracture subcategory of $\Cond(\An)$. We then look at some formal and computational implications thereof.

\subsection{Constructing the fractured structure} \label{subsec:fracture}
We will do this via Lurie's geometric site machinery (\cref{theorem: geometric site}), so we begin by proposing such a site.

\begin{dfn}
 Let $\ExtrDisc^{\open}$ be the site consisting of the wide subcategory of $\ExtrDisc$ with maps given by open embeddings and coverings given by finitely jointly surjective families of open embeddings defining a Grothendieck topology $\tau_{\open}$. We denote by $\Cond^{\open}(\An) \coloneqq \Shv(\ExtrDisc^{\open})$ the resulting $\infty$-category of sheaves.
\end{dfn}

\begin{rem} \label{rem:extrdisc not subcanonical}
   Let us note here that, unlike $\ExtrDisc$, the topology on $\ExtrDisc^{\open}$ is not subcanonical. Indeed, this can already be seen in the simple case of the discrete topological space $\{a,b\}$. It admits a cover by $\{a\}$ and $\{b\}$, but the preferred map 
   \[\Map_{\ExtrDisc^{\open}}(\{a,b\},\{*\}) \longrightarrow \Map_{\ExtrDisc^{\open}}(\{a\},\{*\}) \times \Map_{\ExtrDisc^{\open}}(\{b\},\{*\}) \]
   is only an injection of sets, not a bijection. We hence use the notational convention $\yo_{\open}(-)$ to denote the representable sheaves in $\Cond^{\open}(\An)$, which are by definition the sheafification of the representable presheaves.
\end{rem}

\noindent Before stating the main theorem, we make some elementary point-set topological observations.

\begin{lemma} \label{lemma:limit in extrdisc}
 Let $F\colon I \to \ExtrDisc$ be a diagram. If the limit of this diagram exists in $\ExtrDisc$, then it is given by the limit in $\CHaus$.
\end{lemma}

\begin{proof}
 Let $L_D$ be the limit of $F$ in $\ExtrDisc$ and $L_C$ be the limit of $F$ in $\CHaus$. Then $L_D$ admits a cone to $F$ in $\CHaus$, which, by the universal property of $L_C$, corresponds to a unique continuous map $f\colon L_D \to L_C$. Now, the forgetful functor $\ExtrDisc \to \Set$ admits a left adjoint given by the Stone--\v{C}ech compactification \cite[Theorem VI.9.1]{maclane1998categories}, so the underlying set of $L_D$ is the same as the underlying set of $L_C$ and $f$ is a bijection of sets. Since $L_D, L_C$ are compact Hausdorff spaces, $f$ is a homeomorphism \cite[Theorem 26.6]{munkres2000topology}.
\end{proof}

\noindent This result has the following immediate implication.

\begin{cor} \label{cor:pullbacks in extrdisc}
  The pullback of an open embedding along an arbitrary map is an open embedding in $\ExtrDisc$, and is evaluated as the pullback in $\CHaus$.
\end{cor}

\begin{rem} \label{rem:extrDisc is closed under open subspaces}
    The result in particular relies on the observation that clopen subspaces of extremally disconnected spaces are again extremally disconnected. While an elementary observation, it is in fact the key point-set topological fact without which \cref{cor:pullbacks in extrdisc} does not hold. See \cref{thm:no fibers} for an explicit counterexample.
\end{rem}

\begin{lemma} \label{lemma:retract}
 Open embeddings in $\ExtrDisc$ are closed under retracts.
\end{lemma}

\begin{proof}
  Assume we have a retract diagram
  \[
  \begin{tikzcd}
    A \arrow[r, "i_1"] \arrow[d, "g", swap] \arrow[rr, bend left, equal] & S \arrow[d, hookrightarrow, "f" description] \arrow[r, "p_1"] & A \arrow[d, "g"] \\ 
    B \arrow[r, "i_2", swap] \arrow[rr, bend right, equal] & T \arrow[r, "p_2", swap] & B
  \end{tikzcd},
  \]
  where $f$ is an open embedding. By \cref{cor:pullbacks in extrdisc} it suffices to prove the left hand square is a pullback square.
  
  \medskip \noindent By \cref{lemma:limit in extrdisc}, the pullback is the intersection $B \times_T S$ in $T$. Since $i_2 \circ g = f \circ i_1$ is injective, also $g$ is injective. As $i_1$ and $g$ are injective, the induced map $A \to B \times_T S$ is injective. It suffices to show it is also surjective, then it will follow from \cite[Theorem 26.6]{munkres2000topology} that the map is a homeomorphism.
  
  \medskip \noindent Let $(b,s)$ be in $B \times_T S$. We wish to show that $(b,s)$ is the image of $p_1(s)$ under the map $A \to B \times_T S$. Indeed, we have $g \circ  p_1 (s) = p_2 \circ f (s) = p_2 \circ i_2 (b) = b$ and
   \[f \circ i_1 \circ p_1(s) = i_2 \circ g \circ p_1(s) = i_2 \circ p_2 \circ f(s) = f(s) \]
   which implies $i_1 \circ p_1 (s) = s$, since $f$ is injective. This finishes the proof.
\end{proof}

\begin{thm} \label{thm:main}
    The functor
    \[ \Cond^{\open}(\An) \longrightarrow \Cond(\An) \] from \cref{theorem: geometric site} induces an equivalence onto a fractured subcategory.
\end{thm}

\begin{proof}
    By \cref{theorem: geometric site} we need to prove that $(\ExtrDisc, \ExtrDisc^{\open}, \tau_{\open})$ defines a geometric site. We know that equivalences are open embeddings (evident), open embeddings are stable under pullback (\cref{cor:pullbacks in extrdisc}), open embeddings satisfy the triangle condition (classic exercise) and are closed under retracts (\cref{lemma:retract}). So $(\ExtrDisc, \ExtrDisc^{\open})$ is an admissibility structure.
	
	\medskip \noindent We are left with verifying that $\ExtrDisc^{\open}$ is compatible with $\tau$. Let $\{g_{\alpha} \colon S_{\alpha} \to S \}_{\alpha \in J}$ be a covering sieve in $\ExtrDisc$, meaning there exists a finite subset $J' \subseteq J$ such that $g \coloneqq (g_{\alpha})_{\alpha \in J'} \colon \coprod_{\alpha \in J'} S_{\alpha} \to S$ is surjective. Because $S$ is extremally disconnected, there exists a section $s \colon S \to \coprod_{\alpha \in J'} S_{\alpha}$ (\cref{rem:extrdisc projective}). For $\alpha \in J'$ we write $s_{\alpha} = s|_{s^{-1}(S_{\alpha})} \colon s^{-1}(S_{\alpha}) \to S_{\alpha}$. Then, the composition
    \[ g_{\alpha} \circ s_{\alpha} = (g \circ s)|_{s^{-1}(S_{\alpha})} = (\id_S)|_{s^{-1}(S_{\alpha})} \colon s^{-1}(S_{\alpha}) \to S \]
    is the inclusion map $s^{-1}(S_{\alpha}) \hookrightarrow S$. This is an open embedding and lies in the covering sieve $\{g_{\alpha} \}_{\alpha \in J}$ since it occurs by precomposing maps to $g_{\alpha}$. Moreover, $s^{-1}(S_{\alpha}) \subseteq S$ is clopen and thus extremally disconnected again (\cref{rem:extrDisc is closed under open subspaces}). Furthermore, $\{g_{\alpha} \circ s_{\alpha} \}_{\alpha \in J'}$ is a finite family of maps which are jointly surjective. So it is a $\tau$-covering consisting of admissible morphisms which shows that the admissibility structure is compatible with $\tau$.
\end{proof}

\noindent We end this section with a formal implication of this result.

\begin{cor}
    The $\infty$-topos $\Cond(\An)$ is (co)monadic over $\Cond^{\open}(\An)$. In particular, the functor $j^* \colon \Cond(\An) \to \Cond^{\open}(\An)$ creates all (co)limits.
\end{cor}

\begin{proof}
	This follows from combining \cref{lemma:monadic,thm:main}.
\end{proof}

\subsection{From condensed anima to sheaves of spaces}
While formal implications of the fractured structure are a first good step, by better understanding $\Cond^{\open}(\An)$, we obtain much more powerful results about $\Cond(\An)$.

\medskip \noindent Here we benefit from specific features of the topology on extremally disconnected spaces, 
\begin{rem}\label{rem:extrdisc sheaves}
    Let $S$ be a profinite set. Then we can construct three different sites on open (or clopen) subsets of $S$:
    \begin{itemize}[leftmargin=*]
        \item $(\Open_S, \tau)$: The poset of open subsets of $S$ with the topology of jointly surjective families of maps.
        \item $(\Clopen_S, \tau)$: The poset of clopen subsets of $S$ with the topology of jointly surjective families of maps.
        \item $(\Clopen_S, \tau^{\text{fin}})$: The poset of clopen subsets of $S$ with the topology of finitely jointly surjective families of maps.
    \end{itemize}
    Moreover, the evident inclusion functors 
    \[
        \begin{tikzcd}
            (\Clopen_S, \tau^{\text{fin}}) \arrow[r, hookrightarrow, "\iota_f"] & (\Clopen_S, \tau) \arrow[r, hookrightarrow, "\iota_c"] & (\Open_S, \tau)
        \end{tikzcd}
    \]
    induce an equivalence on $\infty$-categories of sheaves 
    \[
        \begin{tikzcd}
           \Shv(S) = \Shv(\Open_S, \tau) \arrow[r, "\simeq"', "(\iota_c)^*"] &  \Shv(\Clopen_S, \tau) \arrow[r, "\simeq"', "(\iota_f)^*"] & \Shv(\Clopen_S, \tau^{\text{fin}}).
        \end{tikzcd}
    \]
    Indeed, as $S$ is compact, every arbitrary clopen covering of $S$ has a finite subcovering, proving that $(\iota_f)^*$ is an equivalence. Moreover, by \cite[Proposition 3.1.7]{arhangelskijtkachenko2008topologicalgroups}, every totally disconnected locally compact Hausdorff space admits a clopen basis, meaning $\iota_c$ is a basis in the sense of \cite[below Definition A.2]{aoki2023tens}. Hence, by \cite[Corollary A.8]{aoki2023tens}, $(\iota_c)^*$ is an equivalence as well.
\end{rem}

\noindent We need to analyze one last site. 

\begin{lemma} \label{lemma:extrDisc over S is clopenS}
    Let $S$ be an extremally disconnected space. Then, the functor 
    \[ \rho_S\colon (\ExtrDisc^{\open})_{/S} \longrightarrow \Clopen_S, \ (i \colon U \to S) \mapsto i(U) \] is an equivalence of categories.
\end{lemma}

\begin{proof}
 By definition $\Clopen_S$ is a poset. Moreover, as morphisms in $\ExtrDisc^{\open}_{/S}$ are inclusions, the category $\ExtrDisc^{\open}_{/S}$ is a preorder. The result now follows from the fact that there is a functor $\Clopen_S \to (\ExtrDisc^{\open})_{/S}$ sending a clopen subset $U \subseteq S$ to the open embedding $U \to S$, which is necessarily the inverse.
\end{proof}

\begin{prop} \label{prop:condopen over yoS}
    Let $S \in \ExtrDisc$ and let $\pi_S\colon \ExtrDisc^{\open}_{/S} \to \ExtrDisc^{\open}$ be the projection functor. Then, the following functors are equivalences of $\infty$-categories:
    \[
        \begin{tikzcd} 
            \Shv(\Open_S, \tau) \arrow[r, "\simeq"', "(\iota_f \circ \iota_c)^*"] &   \Shv(\Clopen_S, \tau^{\mathrm{fin}}) \arrow[r, "\simeq"', "(\rho_S)^*"] & \Shv(\ExtrDisc^{\open}_{/S}) \arrow[r, "\simeq"', "(\pi_S)_!"] & \Cond^{\open}(\An)_{/\yo_{\open}(S)}
        \end{tikzcd}.
    \]
    Hence, these compose to an equivalence $\eta_S\colon\Shv(S) \xrightarrow{\simeq} \Cond^{\open}(\An)_{/\yo_{\open}(S)}$.
\end{prop}
    
\begin{proof}
    The functor $(\iota_f \circ \iota_c)^*$ is an equivalence by \cref{rem:extrdisc sheaves}. Moreover, $(\rho_S)^*$ is an equivalence, as $\rho_S$ is an equivalence by \cref{lemma:extrDisc over S is clopenS}. Finally, $(\pi_S)_!$ is an equivalence by \cite[Lemma 2.2]{jansen2024stratified}.
\end{proof}

\begin{cor}
    Suppose that $X \in \Cond^{\open}(\An)$ is given by a colimit of representables, i.e.~via an equivalence $f\colon\colim_i \yo_{\open} S_i \xrightarrow{\simeq} X$. Then,
    \[ \Cond^{\open}(\An)_{/X} \xrightarrow[\simeq]{f^*}  \lim_i \Cond^{\open}(\An)_{/\yo_{\open} S_i}  \xrightarrow[\simeq]{(\eta_{S_i})^{-1}} \lim_i \Shv(S_i) \]
    is an equivalence of $\infty$-categories.
\end{cor}

\begin{proof}
    The first equivalence is an immediate consequence of descent in $\infty$-topoi \cite[Theorem 6.1.3.9]{lurie2009htt} and the second follows from  \cref{prop:condopen over yoS}.
\end{proof}

\noindent We now use these computations to give an explicit collection of enough points of $\Cond(\An)$. The mere existence of enough points already follows from Deligne's completeness theorem \cite[Theorem A.4.0.5]{lurie2018sag}, see \cite[Corollary 2.4.5]{barwickhaine2019pyknotic}. However, we are unaware of any explicit characterization of these points in the literature.

\begin{lemma} \label{lemma:shvs hypercomplete}
	Let $S \in \ProFin$. Then, $\Shv(S)$ is hypercomplete.
\end{lemma}

\begin{proof}
	By \cite[Remark 7.2.3.3]{lurie2009htt} $S$ has covering dimension $\leq 0$ since every finite set has covering dimension $\leq 0$. Thus, $\Shv(S)$ has homotopy dimension $\leq 0$ by \cite[Theorem 7.2.3.6]{lurie2009htt}. Since $\Shv(S) \simeq \Shv(\Clopen_S, \tau^{\text{fin}})$ by \cref{rem:extrdisc sheaves}, it is generated under colimits by the representables $\yo_{\open}U$ with $U \in \Clopen_S$. Because $U$ is also a profinite set, $\Shv(S)_{/\yo_{\open} U} \simeq \Shv(U)$ has homotopy dimension $\leq 0$ as well. This shows that $\Shv(S)$ is locally of homotopy dimension $\leq 0$. Jardine's theorem \cite[Corollary 7.2.1.12]{lurie2009htt} implies that it is hypercomplete.
\end{proof}

\begin{thm} \label{thm:formula enough points}
    The collection of points
    \[ \left\{ \Cond(\An) \longrightarrow \An, \ F \mapsto \colim_{\substack{U \subseteq S \ \mathrm{clopen} \\ s \in U}} F(U) \right\}_{S \in \ExtrDisc, \ s \in S} \]
    is jointly conservative.
\end{thm}

\begin{proof}
    The main tool is the fractured structure (\cref{thm:main}). We will provide a collection of jointly conservative points in $\Cond^{\open}(\An)$, which induces a collection of jointly conservative points in $\Cond(\An)$ by \cref{prop: enough points}.
    
    \medskip \noindent First, let $S \in \ExtrDisc$. Then, $\Cond^{\open}(\An)_{/\yo_{\open} S} \simeq \Shv(S)$ by \cref{prop:condopen over yoS}. It is hypercomplete by \cref{lemma:shvs hypercomplete}. So the collection of points $s \colon \An \to \Cond^{\open}(\An)_{/\yo_{\open} S}$ with $s \in S$ is conservative by \cite[Lemma A.3.9]{lurie2017ha}. Then, we obtain a collection of points
    \[
        \begin{tikzcd}
            \An \arrow[r, shift right = 1.5, "s_{*}", swap, ""name = 2] & \Cond^{\open}(\An)_{/\yo_{\open} S} \arrow[l, shift right = 1.5, "s^*", swap, ""name = 1] \arrow[r, swap, shift right = 1.5, ""name = 4] & \Cond^{\open}(\An) \arrow[l, shift right = 1.5, "- \times \yo_{\open} S", swap, ""name = 3]
            \arrow[phantom,from=1,to=2,"\scriptscriptstyle\bot"]
            \arrow[phantom,from=3,to=4,"\scriptscriptstyle\bot"]
        \end{tikzcd}
    \]
    running over $S \in \ExtrDisc$ and $s \in S$. We will now verify that this collection of maps is jointly conservative.
    
    \medskip \noindent Let $f \colon X \to Y$ be a map in $\Cond^{\open}(\An)$ and suppose that $s^*(f \times \yo_{\open} S)$ is an equivalence for all $S \in \ExtrDisc$ and $s \in S$. By conservativity of the $s^*$, we deduce that $f \times \yo_{\open} S$ is an equivalence for all $S \in \ExtrDisc$. Moreover, $\colim_{S \in \ExtrDisc^{\open}} \yo_{\open} S \simeq *$ by \cite[Proposition 6.2.13]{cisinski2019highercategories}. We now have the following diagram 
    \[
        \begin{tikzcd}[column sep=2.5cm]
            \underset{S \in \ExtrDisc^{\open}}\colim (X \times \yo_{\open} S) \arrow[r, "\underset{S \in \ExtrDisc^{\open}}{\colim}(f\times \yo_{\open} S)", "\simeq"'] \arrow[d, "\simeq"'] & \underset{S \in \ExtrDisc^{\open}}{\colim} (Y \times \yo_{\open} S) \arrow[d, "\simeq"]
            \\ X \times * \simeq X \times \underset{S \in \ExtrDisc^{\open}}{\colim} \yo_{\open} S \arrow[r] & Y \times \underset{S \in \ExtrDisc^{\open}}{\colim} \yo_{\open} S \simeq Y \times *
        \end{tikzcd}
    \]
    The top map is an equivalence by our above observations. The vertical maps are equivalences because $\infty$-topoi are Cartesian closed. Hence, the bottom map is an equivalence.
    
    \medskip \noindent By \cref{prop: enough points} the collection of points
    \[
    	\begin{tikzcd}
    		\An \arrow[r, shift right = 1.5, "s_{*}", swap, ""name = 2] & \Cond^{\open}(\An)_{/\yo_{\open} S} \arrow[l, shift right = 1.5, "s^*", swap, ""name = 1] \arrow[r, swap, shift right = 1.5, ""name = 4] & \Cond^{\open}(\An) \arrow[l, shift right = 1.5, "- \times \yo_{\open} S", swap, ""name = 3] \arrow[r, "j_*", swap, shift right = 1.5, ""name = 6] & \arrow[l, shift right = 1.5, "j^*", swap, ""name = 5] \Cond(\An)
    		\arrow[phantom,from=1,to=2,"\scriptscriptstyle\bot"]
    		\arrow[phantom,from=3,to=4,"\scriptscriptstyle\bot"]
    		\arrow[phantom,from=5,to=6,"\scriptscriptstyle\bot"]
    	\end{tikzcd}
    \]
    running over $S \in \ExtrDisc$ and $s \in S$ is jointly conservative. The top composite has the explicit formula claimed above.
\end{proof}

\noindent We can now apply \cref{cor:hypercomplete} for a new proof of hypercompleteness, previously shown in \cite[Corollary 2.4.3]{barwickhaine2019pyknotic}.

\begin{cor} \label{corollary: hypercomplete}
	The $\infty$-topos $\Cond(\An) = \Shv_{\An}(\ExtrDisc)$ is hypercomplete.
\end{cor}

\noindent One may also use the fractured structure to recover results about cohomology.

\begin{cor}
Let \label{corollary: same shape} $S \in \ExtrDisc$. Then, the shape of $\Shv(S)$ and the shape of $\Cond(\An)_{/\yo S}$ coincide and are given by the underlying homotopy type of $S$.
\end{cor}

\begin{proof}
	The geometric morphism $j^* \colon \Cond(\An)_{/\yo S} \to \Cond^{\open}(\An)_{/\yo_{\open}S} \simeq \Shv(S)$ admits a further left adjoint, so it preserves shapes by \cite[Proposition 1.7]{clough2024diff2}. The second part now follows from the fact that the shape of $\Shv(S)$ is the underlying homotopy type of $S$, by \cite[Example 1.2.13]{mair2025thesis}.
\end{proof}

\begin{rem}
    By \cite[Theorem 1.2.18, Corollary 1.2.33]{mair2025thesis}, the shape of $\Cond(\An)_{/\yo X}$ agrees with the shape of $\Shv(X)$ for every paracompact compactly generated space $X$, which is the underlying homotopy type of $X$ in the case of CW complexes \cite[Example 2.4]{hoyois2018highergalois} or profinite sets \cite[Example 1.2.13]{mair2025thesis}. On the other hand, there exists a locally contractible compact Hausdorff space $X$ such that the shape of $\Shv(X)$ does not agree with the underlying homotopy type of $X$, see \cite[Example 2.2.4]{carchedi2021etalehomotopy}.
\end{rem}

\begin{cor} \label{corollary: cohomology}
Condensed and sheaf cohomology over $\mathbb{Z}$ agree for extremally disconnected spaces.
\end{cor}

\begin{proof}
	The topos-theoretic cohomology only depends on the shape, see e.g.~\cite[Construction 1.3.1]{mair2025thesis}. Thus, we are done by \cref{corollary: same shape}.
\end{proof}

\begin{rem}
    Using a variety of more specialized methods, \cref{corollary: cohomology} is also known for compact Hausdorff spaces \cite[Theorem 3.2]{scholze2019condensed}, paracompact compactly generated or locally contractible spaces \cite[Proposition 1.3.5]{mair2025thesis}, or locally compact Hausdorff spaces \cite[Corollary 4.12]{haine2022condensed}.
\end{rem}

\noindent Following \cite[Proposition 2.8, Exercise 1]{clausenscholze2022condensed}, compact Hausdorff spaces are compact in $\Cond(\Set)$. An $\infty$-categorical lift of this result for extremally disconnected spaces can be seen directly by a Yoneda-type argument after checking that filtered colimits in $\Cond(\An)$ are already computed in $\Fun(\ExtrDisc^{\op}, \An)$ using \cref{lemma: condensed anima is finite product preserving functor}. Nonetheless, the fractured structure provides an alternative general method of attack.

\begin{cor} \label{corollary: compact}
	Extremally disconnected spaces are compact in $\Cond(\An)$.
\end{cor}

\begin{proof}
	After noting $\Shv(S) \simeq \Cond^{\open}(\An)_{/\yo_{\open} S}$ by \cref{prop:condopen over yoS} and that $\yo_{\open} S$ is compact therein by \cite[Remark 7.3.1.5, Corollary 7.3.4.11]{lurie2009htt}, we can run the same proof as in \cite[Theorem 2.15]{clough2024diff1}.
\end{proof}

\section{Further candidates for fractured structures}
\noindent In \cref{sec:fractured} we observed a fractured structure via extremally disconnected spaces and open embeddings. However, condensed objects can be realized via different sites, and each comes with different notions of ``inclusions". So, this leaves us with the question whether there are alternative options.

\medskip \noindent In this section we observe that further potential candidates will in fact not satisfy the desired conditions. Interestingly, this requires us to venture into technical aspects of point-set topology.

\subsection{Compact Hausdorff spaces, profinite sets and (open) embeddings}
A natural enlargement of $\ExtrDisc^{\open}$ is achieved by starting with a larger collection of objects, i.e.,~by starting with one of the sites $\CHaus$ or $\ProFin$, and by attempting to equip these with a geometric site structure. Recall that
\[ \Cond(\An) = \Shv(\ExtrDisc) \simeq \Shv(\ProFin)^{\mathrm{hyp}} \simeq \Shv(\CHaus)^{\mathrm{hyp}}, \]
so we furthermore have to be careful with hypercompletions here.

\begin{rem}\label{rem:hyp truncated}
   While it is generally difficult to determine whether an arbitrary sheaf on the categories $\ProFin$ resp.~$\CHaus$ is hypercomplete, following \cite[Lemma 6.5.2.9]{lurie2009htt} every truncated sheaf is hypercomplete.
\end{rem}

\noindent The first step lies in constructing admissibility structures. Consider the wide subcategories
\[ \CHaus^{\inj}, \CHaus^{\open} \subset \CHaus \quad \text{and} \quad \ProFin^{\inj}, \ProFin^{\open} \subset \ProFin \]
spanned by injective maps resp.~open embeddings. These can be endowed with the Grothendieck topology $\tau_{\inj}$ resp.~$\tau_{\open}$ with coverings given by finitely jointly surjective families of injective maps resp.~open embeddings. A similar (but less involved) argument as in \cref{thm:main} yields:

\begin{lemma}
	The wide subcategories
	\[ \CHaus^{\inj}, \CHaus^{\open} \subset \CHaus \quad \text{and} \quad \ProFin^{\inj}, \ProFin^{\open} \subset \ProFin \]
	each define admissibility structures.
\end{lemma}

\noindent The natural hope is that they refine to geometric sites with $\tau_{\inj}$ resp.~$\tau_{\open}$. The purpose of this subsection is to discuss that the admissibility structures are not compatible with $\tau_{\inj}$ resp.~$\tau_{\open}$.

\begin{dfn} \label{def: finite collection of jointly surjective local sections}
	Let $f \colon X \to Y$ be a map of compact Hausdorff spaces. It is said to \emph{admit a finite collection of jointly surjective local sections} if there exists a finite collection of closed\footnote{As we are looking at morphisms in $\CHaus$, we consider closed subspaces.} subspaces $Z_1, \cdots, Z_n \subseteq Y$ with maps $g_i \colon Z_i \to X$ such that $f \circ g_i = \id_{Z_i}$ for all $i$ and $\bigcup_{i=1}^n Z_i = Y$.
\end{dfn}

\noindent Recall from \cite[Theorem VI.9.1]{maclane1998categories} that the forgetful functor $U\colon\CHaus \to \Set$ admits a left adjoint given by Stone--\v{C}ech compactification $\beta$ which factors over the forgetful functor $\ExtrDisc \to \CHaus$. For more details regarding the topological space $\beta S$, see \cref{rem:stonecech}.

\medskip \noindent Now, we consider $\bN$ with the discrete topology and the inclusion into its one-point compactification $\bN \to \bN \cup \{\infty \}$. The adjunction yields a continuous map $e \colon \beta \bN \to \bN \cup \{ \infty \}$. Before we get to the main result, let us make the following technical observation regarding surjective maps. 

\begin{rem} \label{rem:e effective epi}
  Let $p\colon Y \to X$ be a surjective map in $\ProFin$. Then, $\yo(p)$ is an effective epimorphism in $\Shv(\ProFin)$. Indeed, the sieve generated by $p$ is a covering sieve. So, by \cite[Proposition 6.2.3.20]{lurie2009htt} in the particular case where $\sX = \Shv(\ProFin)$, $\sC = \ProFin$, $f = \yo\colon \ProFin \to \Shv(\ProFin)$, and the fact that $\yo$ is the image of the identity functor, $\yo(p)$ is an effective epimorphism.
\end{rem}

\noindent We now have the following observation about fractured structures. 

\begin{prop}\label{prop:betaN admits finite collection}
    Assume that any of the following inclusions of sites
    \[ \CHaus^{\inj}, \CHaus^{\open} \subset \CHaus \quad \text{and} \quad \ProFin^{\inj}, \ProFin^{\open} \subset \ProFin \] 
    induces a fractured structure via
    \[
    	\begin{tikzcd}
    		\Cond^{\inj}(\An) \coloneqq \Shv(\ProFin^{\inj}) \arrow[r] & \Shv(\ProFin) \arrow[r] & \Cond(\An)
    	\end{tikzcd}
    \]
    (and similarly the other variants). Then $e\colon \beta \bN \to \bN \cup \{ \infty \}$ admits a finite collection of jointly surjective local sections.
\end{prop}
        
\begin{proof}
  We will give the argument for $\ProFin$; the exact same argument works for $\CHaus$ by replacing $\ProFin$ with $\CHaus$ everywhere.
  
  \medskip \noindent Consider the map $e \colon \beta \bN \to \bN \cup \{\infty \}$ of profinite sets. As a map of compact Hausdorff spaces, its image is closed. It hits all of $\bN$ by construction, so it must also hit $\infty$. Hence, $e$ is surjective. Now consider the map $\yo(e)$. As the site is subcanonical, $\yo(e)$ is a map of sheaves. Moreover, $\yo(e)$ is a map of $0$-truncated sheaves, and so is also a map of hypersheaves, meaning it lives in $\Cond(\An)$ by \cref{rem:hyp truncated}. Finally, as $e$ is surjective, $\yo(e)$ is an effective epimorphism 
  in $\Cond(\An)$ by \cref{rem:e effective epi}.
  
  \medskip \noindent As $j^*$ preserves colimits and limits, it preserves effective epimorphisms, so 
  \[ j^*(\yo(e)) \colon j^*(\yo(\beta \bN)) \longrightarrow j^*(\yo(\bN \cup \{\infty\})) \] 
  is an effective epimorphism in $\Cond^{\mathrm{inj}}(\An)$.  Let us temporarily write $\yo_{\inj}$ for the Yoneda embedding on $\ProFin^{\inj}$ to distinguish it from the Yoneda embedding on $\ProFin$, and denote its sheafification by $\yo_{\inj,\mathrm{sh}}$\footnote{Analogous to \cref{rem:extrdisc not subcanonical} this presheaf need not be a sheaf, as the site on $\ProFin^{\inj}$ is not subcanonical.}. We consider the following pullback square of presheaves on $\ProFin^{\mathrm{inj}}$:
  \[
    \begin{tikzcd}
      R \arrow[r, "l"] \arrow[d, "p", swap] \arrow[dr, phantom, very near start, "\lrcorner"] & P \arrow[r] \arrow[d, "q" description] \arrow[dr, phantom, very near start, "\lrcorner"] & j^*(\yo(\beta\bN)) \arrow[d, "j^*(\yo(e))"] \\ 
      \yo_{\inj}(\bN \cup \{\infty\}) \arrow[r, "i", swap] & \yo_{\inj,\mathrm{sh}}(\bN \cup \{\infty\}) \arrow[r] & j^*\yo(\bN \cup \{\infty\})
    \end{tikzcd}.
    \]
    Here the right hand square is in fact a pullback square of sheaves. By pullback stability of effective epimorphisms in $\Cond^{\mathrm{inj}}(\An)$, $q$ is an effective epimorphism of sheaves as well. Moreover, by construction, $i$ is a local equivalence of presheaves (i.e., becomes an equivalence after sheafification). By left exactness of sheafification \cite[Corollary 6.2.1.6, Proposition 6.2.2.7]{lurie2009htt}, $l$ is hence also a local equivalence of presheaves. This means that the sheafification of $p$ is the effective epimorphism $q$.
    \medskip \\Let us consider the image factorization of $p$ inside $\Fun((\ProFin^{\mathrm{inj}})^{\op},\An)$, and denote it by
    \[
    	\begin{tikzcd}
    		R \arrow[r, twoheadrightarrow] & S \arrow[r, hookrightarrow] & \yo_{\inj}(\bN \cup \{\infty \}).
    	\end{tikzcd}
    \] 
    Since $q$ is an effective epimorphism, the second map $S \hookrightarrow \yo_{\inj}(\bN \cup \{\infty\})$ becomes an equivalence after sheafification. This means that $S$ is a covering sieve of $\bN \cup \{\infty\}$, see \cite[Lemma 6.2.2.16]{lurie2009htt}.
    
    \medskip \noindent Unwinding the pullback square, $S \hookrightarrow \yo_{\inj}(\bN \cup \{\infty\})$ precisely consists of injections into $\bN \cup \{\infty\}$ that admit a lift to $\beta \bN$. By definition of the topology on $\ProFin^{\mathrm{inj}}$, this means there is a finite family of injections $k_\alpha \colon V_\alpha\hookrightarrow \bN \cup \{\infty\}$ that are jointly surjective and such that each $k_\alpha$ admits a lift to $\beta\bN$. This is precisely the condition that $e$ admits a finite collection of jointly surjective local sections.
\end{proof}

\begin{lemma} \label{lemma:no finite collection of jointly surjective local sections}
 The map $e \colon \beta \bN \to \bN \cup \{\infty\}$ does not admit a finite collection of jointly surjective local sections.
\end{lemma}

\begin{proof} 
    Suppose that a finite collection of jointly surjective local sections exists, i.e.,~closed subspaces $Z_1, \cdots, Z_n \subseteq \bN \cup \{\infty \}$ with maps $g_i \colon Z_i \to \beta \bN$ satisfying the properties in \cref{def: finite collection of jointly surjective local sections}. Suppose that $Z_i$ is an infinite set for some $i$. In particular, it must contain $\infty$, since it contains a sequence of strictly increasing natural numbers which converges to $\infty$. By assumption, $g_i \colon Z_i \to \beta \bN$ is a closed embedding of compact Hausdorff spaces. Thus, $\beta \bN$ must also have a convergent sequence of pairwise distinct elements. However, this is not possible, since $\beta \bN$ is extremally disconnected (\cref{remark: ExtrDisc categorically bad}). 
\end{proof}

\noindent Combining \cref{prop:betaN admits finite collection} and \cref{lemma:no finite collection of jointly surjective local sections} we obtain the following result.

\begin{cor} 
    The admissibility structures \label{corollary: no fractured structure CHaus ProFin}
    \[ \CHaus^{\inj}, \CHaus^{\open} \subset \CHaus \quad \text{and} \quad \ProFin^{\inj}, \ProFin^{\open} \subset \ProFin \]
    do not induce a fractured structure on $\Cond(\An)$.
\end{cor}

\begin{rem}
	The argument furthermore shows that the suggested constructions do not yield fractured structures on $\Shv(\ProFin)$ and $\Shv(\CHaus)$; in fact the associated right adjoints do not preserve effective epimorphisms. This is the only feature about fractured structures that we have used in the proof of \cref{prop:betaN admits finite collection}.
\end{rem}

\begin{rem}
	The surjective map
	\[ [0,1]^{\bN} \to (S^1)^{\bN}, \ (x_0, x_1, \cdots) \mapsto \left(e^{2\pi i x_0}, e^{2\pi i x_1}, \cdots \right) \]
	also does not admit a finite collection of jointly surjective local sections. This would have been enough to discuss only the $\CHaus$ case. 
\end{rem}

\noindent The relevance of local sections was a motivating factor for us to consider extremally disconnected spaces, as those are the projective objects in $\CHaus$ (\cref{rem:extrdisc projective}).

\subsection{Extremally disconnected spaces and embeddings} \label{subsection: ExtrDisc counterexample}
Staying with the site of extremally disconnected spaces, we can ask whether we can relax the condition of open embeddings to embeddings. This is a natural question because the latter are more general and still have good categorical properties. However, we will see that this also fails. Concretely, we want to see that pullbacks along injections need not exist in $\ExtrDisc$, by proving the following conjecture, suggested by Dustin Clausen.

\begin{conj}[Clausen] \label{conj:no fibers in extrdisc}
    The category $\ExtrDisc$ does not admit all fibers.
\end{conj}

\noindent Before we can state the main result, we establish some notation and recall some facts about the Stone--\v{C}ech compactification that will be relevant throughout. A \emph{filter} on a set $S$ is a non-empty collection of subsets of $S$ which is closed under finite intersections and supersets. An \emph{ultrafilter} on $S$ is a filter, such that for all $A \subseteq S$, it either contains $A$ or $S \setminus A$. For a given element $s \in S$, the collection $\{A \subseteq S\colon s \in A\}$ is an example of an ultrafilter, called a \emph{principal ultrafilter}. We only review the following result about ultrafilters.

\begin{lemma}[Tarski's Theorem {\cite[Theorem 7.5]{hrbacekjech2000introduction}}] \label{lemma:non-principal ultrafilters}
    Let $S$ be a set. Then every filter on $S$ can be extended to an ultrafilter. Moreover, the resulting ultrafilter is non-principal if and only if it contains all cofinite subsets. 
\end{lemma}

\noindent Let us recall some detailed aspects of the Stone--\v{C}ech compactification $\beta S$ of a set $S$.

\begin{rem} \label{rem:stonecech}
    To each set $S$, we can associate the \emph{Stone--\v{C}ech compactification} $\beta S$. The underlying set of $\beta S$ is the set of ultrafilters, and the topology is generated by the sets $A^* \coloneqq \{u \in \beta S : A \in u\}$ for $A \subseteq S$. Furthermore, there is a (continuous) map $S \to \beta S$, sending $s$ to the principal ultrafilter $u_s\coloneqq\{A \subseteq S : s \in A\}$.
\end{rem}

\noindent We now have the following main result.

\begin{thm} \label{thm:no fibers}
 Let $p$ be a non-principal ultrafilter on $\bN$. Let $\pi_1 \colon \bN \times \bN \to \bN$ denote the projection to the first factor. Then, the map $\beta \pi_1 \colon \beta(\bN \times \bN) \to \beta \bN$ does not admit a fiber over $p$ in $\ExtrDisc$.
\end{thm}

\noindent The consequence is that injections on $\ExtrDisc$ do not lend themselves to a fractured structure.

\begin{cor} \label{corollary: ExtrDiscinj not admissibility}
	The injections on $\ExtrDisc$ do not form an admissibility structure.
\end{cor}

\noindent The proof requires further delving into the technicalities of the Stone--\v{C}ech compactification. Let us recall the following lemma. See \cite[Theorem 3.18(b)]{hindmanstrauss2012stonecech} for further details.

\begin{lemma} \label{lemma:clopen}
 Let $S$ be a set. Then a subset of $\beta S$ is clopen if and only if it is of the form $A^*$, for some $A \subseteq S$.
\end{lemma}

\noindent Let $p$ be an ultrafilter on $\bN$. We denote the pre-image (fiber) of $p$ via $\beta \pi_1\colon \beta(\bN \times \bN) \to \beta \bN$ by $F_p$. Note that for any $p$, $F_p$ is a compact Hausdorff space, and a closed subspace of $\beta(\bN \times \bN)$. We have the following lemma regarding $F_p$, which is an immediate unwinding of definitions.

\begin{lemma} \label{lemma:description of fp}
 Let $p$ be an ultrafilter on $\bN$ and $u$ an ultrafilter on $\bN \times \bN$. Then,
 \[
   u \in F_p \iff (\text{for all }A \subseteq \bN: A \in p \iff A \times \bN \in u).
 \]
\end{lemma}

\begin{proof}
 By definition $u \in F_p$ if and only if $\beta \pi_1(u) = \{ A \subseteq \bN \mid  A \times \bN \in u \} = p$. This equality precisely states $A \in p$ if and only if $A \times \bN \in u$. 
\end{proof}

\noindent We now establish some notational conventions. Let $J \subseteq \bN \times \bN$ be a subset. The \emph{$k$-th slice} of $J$, denoted by $\Sl_k(J)$, is the subset $\{n \in \bN : (n,k) \in J\}$. Let $p$ be an ultrafilter on $\bN$ and $C_k = F_p \cap (\bN \times \{k\})^*$. Following \cref{lemma:clopen}, the $C_k$ are clopen subsets of $F_p$. They are related to the slices via the following result.

\begin{lemma} \label{lemma:slices}
 Let $p$ be an ultrafilter on $\bN$ and $J \subseteq \bN \times \bN$. For $k \in \bN$ we have $C_k \subseteq F_p \cap J^*$ if and only if $\Sl_k(J) \in p$.
\end{lemma}

\begin{proof}
  First, let $\Sl_k(J) \in p$. Let $u \in C_k = F_p \cap (\bN \times \{k\})^*$. Then, by definition, $u \in F_p$ and $\bN \times \{k\} \in u$. By \cref{lemma:description of fp}, since $\Sl_k(J) \in p$ and $u \in F_p$, we have $\Sl_k(J) \times \bN \in u$. As filters are closed under finite intersections, we have
\[ \Sl_k(J) \times \{k\} = (\Sl_k(J) \cap \bN) \times (\bN \cap \{k\}) = (\Sl_k(J) \times \bN) \cap (\bN \times \{k\}) \in u. \]
But, by definition of slices, $\Sl_k(J) \times \{k\} \subseteq J$, so by upward closure we also have $J \in u$ or, equivalently,  $u \in J^*$. Hence $u \in F_p \cap J^*$, and we are done.

\medskip \noindent Now, on the other side, suppose $\Sl_k(J) \notin p$. We construct an ultrafilter in $C_k$ that is not in  $F_p \cap J^*$. Since $p$ is an ultrafilter, $B \coloneqq \bN \setminus \Sl_k(J) = \{n \in \bN \mid (n,k) \notin J\} \in p$. Note that by construction 
\begin{equation} \label{eq:intersection empty}
    (B \times \{k\}) \cap J = \emptyset.
\end{equation}

\noindent Consider the family of subsets of $\bN \times \bN$:
\[ \cF \coloneqq \{A \times \bN : A \in p\} \cup \{\bN \times \{k\}, B \times \{k\} \}. \]

\noindent Observe that $\cF$ is closed under finite intersections. Indeed, we only need to consider intersections of the form $(A \times \bN) \cap (\bN \times \{k\})$ and $(A \times \bN) \cap (B \times \{k\})$. We only consider the second case, and the first is analogous. Here, for $A \in p$, we have
\[ (A \times \bN) \cap (B \times \{k\}) = (A \cap B) \times \{k\} \neq \emptyset \]
since $A \cap B \in p$, and hence non-empty. Therefore, $\cF$ has the finite intersection property and extends to an ultrafilter $u_\cF$ on $\bN \times \bN$ by Tarski's theorem (\cref{lemma:non-principal ultrafilters}).

\medskip \noindent By construction, $u_\cF$ contains $A \times \bN$ for all $A \in p$, so $u_\cF \in F_p$ by \cref{lemma:description of fp}. Also $\bN \times \{k\} \in u_\cF$, so $u_\cF \in C_k = F_p \cap (\bN \times \{k\})^*$. On the other hand, $B \times \{k\} \in u_\cF$ and $(B \times \{k\}) \cap J = \emptyset$  by \cref{eq:intersection empty}, so $J \notin u_\cF$, in other words, $u_\cF \notin J^*$. Thus $u_\cF \in C_k$ and $u_\cF \notin F_p \cap J^*$, and we are done.
\end{proof}

\noindent We now fix one open subset in $F_p$.

\begin{notn} \label{not:up}
 We will denote the union $\bigcup_{k \in \bN} C_k$ by $U_p \coloneqq \bigcup_{k \in \bN} C_k$.
\end{notn}

\noindent It is immediate that $U_p$ is open in $F_p$, as a union of the open subsets $C_k$. We are now ready to state the crucial lemma.

\begin{lemma}
 Let $p$ be a non-principal ultrafilter on $\bN$. Then for every clopen subset $V \subseteq F_p$, such that $U_p \subseteq V$, there exists a clopen subset $V'$ of $F_p$, such that $U_p \subseteq V' \subsetneq V$.
\end{lemma}

\begin{proof}
 Let $V = F_p \cap J^*$ for some $J \subseteq \bN \times \bN$, such that $U_p \subseteq V$. Then, $\Sl_k(J) \in p$ for all $k \in \bN$ by \cref{lemma:slices}. By \cref{lemma:non-principal ultrafilters}, we have $T_k \coloneqq \{n \in \bN \mid n \geq k\} \in p$ for all $k \in \bN$. This means for all $k \in \bN$ that
 \[B_k \coloneqq T_k \cap \bigcap_{0 \leq j \leq k} \Sl_j(J) \in p,\]
 as filters are closed under finite intersections. Note that $\bigcap_{k \in \bN} B_k \subseteq \bigcap_{k \in \bN} T_k = \emptyset$. Now, for all $k \in \bN$, let $E_k \coloneqq B_k \setminus B_{k+1}$. These sets have the following properties:
 \begin{itemize}[leftmargin=*]
    \item First, $E_k \subseteq B_k \subseteq \Sl_k(J)$, so, by definition of slices, $E_k \times \{k\} \subseteq J$.
    \item Next, $E_k \not\in p$, as $B_{k+1} \in p$ and $B_{k+1} \cap E_k = \emptyset$, and an ultrafilter cannot contain the empty set.
    \item Lastly, $\bigcup_{k \in \bN} E_k = \bigcup_{k \in \bN} B_k \setminus B_{k+1} = B_0 \in p$.
 \end{itemize}  
 Now, let $R \coloneqq \bigcup_{k \in \bN} (E_k \times \{k\}) \subseteq J$, and write $J' \coloneqq J \setminus R$. By definition, $V' \coloneqq F_p \cap (J')^*$ is clopen in $F_p$. To finish the proof, we show that $U_p \subseteq V'$ and $V' \subsetneq V$.
 
 \medskip \noindent We start by checking $U_p \subseteq V'$. Fix a $k \in \bN$. By construction $\Sl_k(J') = \Sl_k(J) \setminus E_k$. Since $E_k \not\in p$, we have 
 \[ \Sl_k(J') = \Sl_k(J) \setminus E_k  = \Sl_k(J) \cap (\bN \setminus E_k) \in p. \] Hence, $U_p \subseteq V'$ by \cref{lemma:slices}.
 
 \medskip \noindent Now, we discuss $V' \subsetneq V$. By definition $J' = J \setminus R$ and so 
 \[V' = F_p \cap (J')^* = F_p \cap (J \setminus R)^* \subseteq (F_p \cap J^*) \setminus (F_p \cap R^*) \subseteq F_p \cap J^* = V.\] 
 From the previous line it also follows that the inclusion is not an equality precisely when $F_p \cap R^* \neq \emptyset$.
 
 \medskip \noindent Let $\cF = \{A \times \bN \mid A \in p\} \cup \{R\}$. Observe that $\cF$ is closed under finite intersections. Indeed, we only need to consider intersections of the form $(A \times \bN) \cap R$. However, we have the computation
  \[\pi_1((A \times \bN) \cap R) = A \cap \pi_1(R) = A \cap \bigcup_{k \in \bN} E_k = A \cap B_0 \in p.\]
 As $\pi_1((A \times \bN) \cap R)$ is non-empty, it hence follows that $(A \times \bN) \cap R$ is non-empty as well. So we may extend $\cF$ to an ultrafilter $u_\cF$ on $\bN \times \bN$ by Tarski's theorem (\cref{lemma:non-principal ultrafilters}). By \cref{lemma:description of fp}, $u_\cF \in F_p$, and, by definition, $R \in u_\cF$, meaning $u_\cF \in R^*$. Hence, $u_\cF \in F_p \cap R^*$ and we are done.
\end{proof}

\noindent This lemma and the universal property of closures immediately has the following implication.

\begin{cor} \label{cor:closure not open}
    Let $p$ be a non-principal ultrafilter on $\bN$. Then $\overline{U}_p$ (\cref{not:up}) is not open in $F_p$.
\end{cor}

\noindent We can now combine all these steps, to conclude the proof of \cref{thm:no fibers}.

\begin{proof}[Proof of \cref{thm:no fibers}]
    By \cref{lemma:limit in extrdisc}, if the fiber exists, then it needs to be the fiber in $\CHaus$. Hence, it suffices to show that the compact Hausdorff space $F_p = (\beta \pi_1)^{-1}(p)$ is not extremally disconnected. Now, by \cref{cor:closure not open}, there is an open subset $U_p \subseteq F_p$, such that $\overline{U}_p$ is not open in $F_p$, thus $F_p$ is not extremally disconnected.
\end{proof} 

\begin{rem} \label{remark: nonprincipal crucial}
  The assumption that $p$ is non-principal is crucial. Indeed, let us choose a principal ultrafilter $p$ on $m \in \bN$. Then, by \cref{lemma:description of fp}, an ultrafilter $u$ is in the fiber $F_p$ if it contains $A \times \bN$, for $A \in p$, which in this case reduces to containing $\{m\} \times \bN$. This means the fiber is precisely $(\{m\}\times\bN)^*$. By \cref{lemma:clopen}, $(\{m\}\times\bN)^*$ is clopen in $\beta(\bN \times \bN)$, and so it is again extremally disconnected (\cref{rem:extrDisc is closed under open subspaces}). 
\end{rem}
  
  \noindent In fact, we can describe the fiber more concretely. Let $(\{m\},\id)\colon \bN \to \bN \times \bN$ be the map sending $n$ to $(m,n)$. 

\begin{prop}
 The map $\beta(\{m\},\id)\colon \beta \bN \to \beta (\bN \times \bN)$ is an injection with image the fiber of $\beta\pi_1$ over the principal ultrafilter on $m$. Hence, the fiber is homeomorphic to $\beta\bN$.
\end{prop}  

\begin{proof}
 First of all, the map $\beta(\{m\},\id)\colon \beta \bN \to \beta (\bN \times \bN)$ is an injection, as it is the section of $\beta \pi_1$. Now, let $u_0$ be any ultrafilter on $\bN$. Then,
  \begin{align*}
 	\beta(\{m\},\id)(u_0) & = \{ A \subseteq \bN \times \bN \mid (\{m\},\id)^{-1}(A) \in u_0\} \\ 
 	& =\{ A \subseteq \bN \times \bN \mid \{k \in \bN \mid (m,k) \in A\} \in u_0\}.
 \end{align*}  
 This contains $\{m \} \times \bN$. So by \cref{remark: nonprincipal crucial} the image of $\beta(\{m \}, \id)$ is contained in the fiber of $\beta \pi_1$ over the principal ultrafilter on $m$. We now check that it is surjective. Let $u$ be an ultrafilter that contains $\{m \} \times \bN$. Setting $u_0 \coloneqq \{B \subseteq \bN \mid \{m\} \times B \in u\}$, it is non-empty because it contains $\bN$, and it evidently satisfies the remaining conditions of an ultrafilter. Using the above computation of $\beta(\{m \}, \id)(u_0)$ we obtain 
 \[ \beta(\{m\},\id)(u_0) = \{ A \subseteq \bN \times \bN \mid \{k \in \bN \mid (m,k) \in A\} \in u_0\} = u, \]
 meaning $u$ is in the image of $\beta(\{m \}, \id)$. Finally, it follows from \cite[Theorem 26.6]{munkres2000topology} that $\beta(\{m\},\id)$ is a homeomorphism onto its image.
\end{proof}

\noindent Let us end with one additional observation regarding limits in $\ExtrDisc$, which is not required for our aims, but can be of independent interest. Besides \cref{cor:pullbacks in extrdisc}, it is the only limit we can obtain in $\ExtrDisc$. Here we fundamentally rely on \cite{givanthalmos2009boolean}.

\begin{prop}
 The category $\ExtrDisc$ is closed under intersection. 
\end{prop}

\begin{proof}
 By \cite[Theorem 39]{givanthalmos2009boolean}, the Stone duality between profinite sets and Boolean algebras restricts to a duality between extremally disconnected spaces and complete Boolean algebras, meaning every subset has both a supremum and an infimum. Moreover, by \cite[Theorem 37]{givanthalmos2009boolean}, closed subspaces correspond to quotients of the Boolean algebra, and, by \cite[Exercise 24.12]{givanthalmos2009boolean}, quotients by complete ideals correspond to extremally disconnected subspaces. Additionally, by \cite[Theorem 33]{givanthalmos2009boolean}, the quotient by the join of two ideals corresponds to the intersection of the corresponding subspaces. Finally, by \cite[Theorem 21]{givanthalmos2009boolean}, the join of two complete ideals is again a complete ideal. Hence, we are done.
\end{proof}

\bibliographystyle{alpha}
\bibliography{main}

\end{document}